\documentclass[12pt,reqno]{article}

\usepackage{amssymb}
\usepackage{amsmath}
\usepackage{amsthm}
\usepackage{amsfonts}
\usepackage{url}
\usepackage{hyperref}
\usepackage{breakurl}

\newcommand{\ise}{ize}
\newcommand{\C}{{\mathbb C}}

\newcommand{\Q}{{\mathbb Q}}
\newcommand{\Z}{{\mathbb Z}}
\newcommand{\Eone}{{E_1}}
\newcommand{\oneFone}{{_1\hspace*{-0.1em}F_1}}
\newcommand{\dup}{\,d}
\newcommand{\raisedot}{\raisebox{2pt}{$.$}}
\newcommand{\raisecomma}{\raisebox{2pt}{$,$}}
\newcommand{\half}{\frac{1}{2}}
\newcommand{\Had}{{\,\odot}}

\theoremstyle{plain}
\newtheorem{theorem}{Theorem}
\newtheorem{corollary}[theorem]{Corollary}
\newtheorem{lemma}[theorem]{Lemma}

\theoremstyle{definition}

\newtheorem{conjecture}[theorem]{Conjecture}

\theoremstyle{remark}
\newtheorem{remark}[theorem]{Remark}

\newcommand{\seqnum}[1]{\href{http://oeis.org/#1}{\underline{#1}}}

\begin{document}
\bibliographystyle{plain}
\title{A Conjectured Integer Sequence Arising
{From} the Exponential Integral}
\author{Richard P. Brent\\
Australian National University\\
Canberra, ACT 2601, Australia\\
\href{mailto:ei@rpbrent.com}{\tt ei@rpbrent.com}
\and
M. L. Glasser\\
Clarkson University\\
Potsdam, NY 136-99-5820, USA\\
\href{mailto:lglasser@clarkson.edu}{\tt lglasser@clarkson.edu}
\and
Anthony J. Guttmann\\
School of Mathematics and Statistics\\
The University of Melbourne\\
Vic.\ 3010, Australia\\
\href{mailto:guttmann@unimelb.edu.au}{\tt guttmann@unimelb.edu.au}
}

\date{}

\maketitle

\begin{abstract}
Let $f_0(z) = \exp(z/(1-z))$,
$f_1(z) = \exp(1/(1-z))E_1(1/(1-z))$,
where $E_1(x) = \int_x^\infty e^{-t}t^{-1}{\dup}t$.
Let $a_n = [z^n]f_0(z)$ and $b_n = [z^n]f_1(z)$ be the corresponding
Maclaurin series coefficients. We show that $a_n$ and $b_n$ may be
expressed in terms of confluent hypergeometric functions.

We consider the asymptotic behaviour of the sequences $(a_n)$ and $(b_n)$
as $n \to \infty$, showing that they are closely related, and proving a 
conjecture of Bruno Salvy regarding $(b_n)$.

Let $\rho_n = a_n b_n$, so
$\sum \rho_n z^n = (f_0\Had f_1)(z)$ 
is a Hadamard product.
We obtain an asymptotic expansion
$2n^{3/2}\rho_n \sim -\sum d_k n^{-k}$ as $n \to \infty$, 
where $d_k\in\Q$, $d_0=1$.
We conjecture that $2^{6k}d_k \in \Z$. 
This has been verified for $k \le 1000$.

\end{abstract}

\pagebreak[3]
\section{Introduction}				\label{sec:intro}
	
We consider two analytic functions,
\[f_0(z) := e^{z/(1-z)} = e^{-1}\,e^{1/(1-z)}\]
and
\[f_1(z) := e^x\Eone(x), 
\text{ where }
x := 1/(1-z)
\text{ and } 
\Eone(x) := \int_x^\infty \frac{e^{-t}}{t}{\dup}t.
\]
These functions are regular in the open disk $D = \{z\in\C: |z| < 1\}$.
We write their Maclaurin coefficients
as $a_n := [z^n]f_0(z)$ and $b_n = [z^n]f_1(z)$.
Thus, in the disk $D$, $f_0(z) = \sum_{n\ge 0} a_n z^n$ and
$f_1(z) = \sum_{n\ge 0} b_n z^n$.

The functions $f_0(z)$ and $f_1(z)$ satisfy the same third-order
linear differential equation with polynomial coefficients.
Thus, the sequences $(a_n)$ and $(b_n)$ are D-finite and satisfy
the same recurrence relation (for sufficiently large $n$).

There are several entries in the OEIS related to the rational 
sequence $(a_n)_{n\ge 0}$.
The numerators are OEIS \seqnum{A067764}, 
and the denominators are OEIS \seqnum{A067653}.
The integers $n!a_n$ are given by OEIS \seqnum{A000262} and,
with alternating signs, by OEIS \seqnum{A293125}.
The numbers $(b_n)_{n\ge 0}$
are unlikely to be rational.%
\footnote{In particular,
$b_0 = G$, where $G := e\Eone(1) \approx 0.596$ is the
Euler-Gompertz constant, whose decimal digits are given
by OEIS \seqnum{A073003}. 
We have $b_n = a_nG - a_n'$, where
$a_n' \in \Q$ and $a_n'$ satisfies essentially the same recurrence
as $a_n$, but with different initial conditions.
Clearly $b_n \in \Q$ if and only if $G \in \Q$.
All that is known is that at least one of $\gamma$ and $G$ is
irrational~\cite{Aptekarev,Rivoal}.}

The numbers $a_n$ and $b_n$ may be expressed in terms of confluent
hypergeometric functions.
If $M(a,b,z) = \oneFone(a;b;z) 
$
and $U(a,b,z)$ are standard solutions of Kummer's differential 
equation, then Lemmas~\ref{lem:anM}--\ref{lem:bnU} show that
$a_n = e^{-1}M(n+1,2,1)$ and $b_n = -\Gamma(n)U(n,0,1)$.

We are interested in the asymptotics of $a_n$ and $b_n$ for large $n$.
Perron~\cite{Perron14}, following Fej\'er~\cite{Fejer08},
showed that
\[a_n \sim \frac{e^{2\sqrt{n}}}{2n^{3/4}{\sqrt{\pi e}}}\,\raisedot
\]

Salvy\footnote{Bruno Salvy, email to A.~J.~Guttmann et al., May 28, 2018.}
conjectured that $b_n$ is of order
$e^{-2\sqrt{n}}n^{-3/4}$. We have verified this conjecture.  
In fact,
\[b_n \sim -\frac{\sqrt{\pi e}}{n^{3/4}e^{2\sqrt{n}}}\,\raisedot\]
A function of the form 
$f(n) = \exp(\alpha n^{\theta + o(1)})$ for $\alpha \ne 0$, 
$\theta \in (0,1)$, is called a \emph{stretched exponential}
in the physics/statistics literature 
(the term \emph{sub-exponential} is used in complexity theory).
Thus, $a_n$ and $b_n$ are stretched exponentials,
with $\alpha = \pm 2$ and $\theta = 1/2$.

The motivation for this paper stems from some enumeration problems in
algebraic combinatorics and mathematical physics. Many such problems
involve ordinary generating functions of power series 
\hbox{$A(x)=\sum_{n \ge 0} A_n x^n$} in which $A_n \sim c \mu^n n^g$. 
In such cases, assuming that $g$ is not a negative integer, one can write 
\[A(x) \sim c\,\Gamma(1+g)(1-\mu x)^{-(1+g)}\]
as $x \to 1/\mu$.
However, in recent years there have been a number of examples, such as
$Av(1324)$ pattern-avoiding permutations \cite{CGZ18}, interacting
partially-directed self-avoiding walks \cite{OPR93}, and Dyck paths
enumerated by maximum height \cite{G15}, in which the corresponding
generating function has coefficients behaving as
$B_n \sim c \mu^n \exp(\alpha n^\theta) n^g,$ with $\alpha < 0$. 
The question then arises as
to the asymptotic form of the generating function. The coefficients $b_n$
considered in this paper are of the form just described, with 
$\theta =1/2$, 
and the underlying generating function is found. Corresponding
results for other values of $\theta$ remain to be discussed.

Theorem~\ref{thm:ckdirect} gives complete asymptotic expansions
of $a_n$ and $b_n$.
These may be written as
\[a_n = \frac{F(n^{1/2})}{2n^{3/4}\sqrt{\pi e}}\;\text{ and }\;
  b_n = -\frac{\sqrt{\pi e}}{n^{3/4}}\,F(-n^{1/2}),
\]
where 
$F(x) \sim e^{2x}\sum_{k \ge 0}c_k x^{-k}$,
for certain constants $c_k\in\Q$, $c_0 = 1$.
The $c_k$ may be computed using Theorem~\ref{thm:ckdirect}
or Lemma~\ref{lemma:ck}.

The \emph{Hadamard product} 
$f_0\Had f_1$
of $f_0$ and $f_1$ is the analytic function
defined for $z\in D$ by
\[(f_0\Had f_1)(z) = \sum_{n\ge 0} a_n b_n z^n.\]
The asymptotic expansions of $a_n$ and $b_n$ imply an asymptotic
expansion for $\rho_n := a_n b_n$
of the form
\[
\rho_n \sim - \frac{1}{2n^{3/2}} \sum_{k \ge 0} d_k n^{-k}, 
\]
where $d_k \in \Q$, $d_0 = 1$ (see Corollary~\ref{cor:product}).

A \emph{dyadic rational} is a rational number of the form $p/q$, where
$q$ is a power of two. Let $Q_2 := \{j/2^k: j, k \in \Z\}$
denote the set of dyadic rationals.

We conjecture, from numerical evidence for $k \le 1000$, that $d_k \in Q_2$.
More precisely, defining $r_k := 2^{6k}d_k$, 
Conjecture~\ref{conj:rkinteger} is that $r_k \in\Z$.
Remark~\ref{remark:rkinteger} gives numerical evidence for a slightly
stronger conjecture.
In Theorem~\ref{thm:factorialRkk} we prove the weaker
(but still nontrivial) result that $k!r_k \in \Z$.

In Remark~\ref{remark:Bessel} we mention an analogous (easily proved) result
for modified Bessel functions, where the product $I_\nu(x)K_\nu(x)$ for
fixed $\nu\in\Z$ has an asymptotic expansion whose coefficients are in $\Q_2$.

The connection with confluent hypergeometric (Kummer) functions is discussed
in~\S\ref{sec:GG}, and asymptotic expansions for $a_n$ and $b_n$ are
considered in~\S\ref{sec:asymp_a_b}. In \S\ref{sec:Maclaurin} we mention
various recurrence relations, continued fractions, and closed-form
expressions related to $a_n$ and $b_n$.
Finally, in \S\S\ref{sec:Hadamard}--\ref{sec:dn_rec},
we consider Hadamard products
and discuss the conjecture mentioned above.

Some comments on notation: $f(x) \sim \sum_{k\ge 0}f_k x^{-k}$ 
means that the sum on the
right is an asymptotic series for $f(x)$ in the sense of Poincar\'e. Thus,
for any fixed $m > 0$,
$f(x) = \sum_{k=0}^{m-1} f_k x^{-k} + O(x^{-m})$ as $x \to \infty$.
The letters $j, k, m, n$ always denote integers (except for $n$ in
Remark~\ref{rem:abgen}).
The notation $(x)_n$ for $n \ge 0$
denotes the \emph{ascending factorial}
or \emph{Pochhammer symbol}, 
defined by $(x)_n := x(x+1)\cdots(x+n-1)$.

\section{Connection with hypergeometric functions} \label{sec:GG} 

The numbers $a_n$ and $b_n$ may be expressed in terms of confluent
hypergeometric functions (Kummer functions),
for which we refer to~\cite[\S13.2]{DLMF}.
If $M(a,b,z)$
and $U(a,b,z)$ are standard solutions $w(a,b,z)$ of Kummer's differential 
equation
$zw'' + (b-z)w' - aw = 0$,
then Lemmas~\ref{lem:anM}--\ref{lem:bnU} below
express $a_n$ and $b_n$ in terms of 
$M(n+1,2,1)$ and $U(n,0,1)$.

Kummer~\cite{Kummer} considered
\begin{equation}			\label{eq:KummerM}
M(a,b,z) = \oneFone(a;b;z) = 
  \sum_{k\ge 0} \frac{(a)_k\,z^k}{(b)_k\,k!}\,\raisecomma
\end{equation}
which is undefined if $b$ is zero or a negative integer.
In the case $a \ne b=0$, we can use the solution
\[
zM(a+1,2,z) = \lim_{b\to 0}\, \frac{b}{a}M(a,b,z).
\]
Tricomi~\cite{Tricomi} introduced the function $U(a,b,z)$
as a second (minimal) solution of Kummer's differential equation.
For our purposes it is convenient to use the integral
representation~\cite[(13.4.4)]{DLMF}
(valid for $\Re(a) > 0$, $\Re(z) > 0$)
\begin{equation}			\label{eq:NIST13.4.4}
U(a,b,z)= \frac{1}{\Gamma(a)}\int_0^\infty e^{-zt}
\, t^{a-1} \, (1+t)^{b-a-1} {\dup}t.
\end{equation}

We remark that the
functions $M$ and $U$ satisfy
recurrence relations, known as
``connection formulas''.  For example,
we mention~\cite[(13.3.1) and (13.3.7)]{DLMF},
both (essentially) due to Gauss (see
Erd\'elyi~\cite[\S6.4 and \S6.6]{Erdelyi}):
\begin{small}
\begin{align}		
\label{eq:Kummer13.3.1}
(b-a)M(a-1,b,z)+(2a-b+z)M(a,b,z)-aM(a+1,b,z) =&\; 0,\\ 
\label{eq:Kummer13.3.7}
U(a-1,b,z)+(b-2a-z)U(a,b,z)+a(a-b+1)U(a+1,b,z) =&\; 0.
\end{align}
\end{small}
\vspace*{-10pt}

Lemmas~\ref{lem:anM}--\ref{lem:bnU} express $a_n$ and $b_n$ in
terms of the Kummer functions $M$ and $U$, respectively.
Lemma~\ref{lem:anM} was stated, without proof, by Covo~\cite{Covo}.
\begin{lemma}				\label{lem:anM}
If $n\in\Z$, $n \ge 1$, and $a_n$ is as above, then 
\begin{equation}			\label{eq:anM}
a_n = e^{-1}M(n+1,2,1).
\end{equation}
\end{lemma}

\begin{proof}
If we put $a=n+1$, $b=2$, and $z=1$ in the connection formula 
\eqref{eq:Kummer13.3.1},
we see that $\widetilde{a_n} := e^{-1}M(n+1,2,1)$ satisfies
the same recurrence \eqref{eq:arec1} as~$a_n$.
Thus, to show that $a_n = \widetilde{a_n}$ for all
$n \ge 1$, it is sufficient to show that $a_n = \widetilde{a_n}$ 
for $n \in \{1,2\}$.
Now
\[
\widetilde{a_1} = e^{-1}M(2,2,1) = e^{-1}\sum_{k\ge 0}\frac{(2)_k}{(2)_k\,k!}
	= 1 = a_1,
\]
and, similarly,
\[
\widetilde{a_2} = e^{-1}M(3,2,1) = e^{-1}\sum_{k\ge 0}\frac{(3)_k}{(2)_k\,k!}
	= e^{-1}\sum_{k \ge 0}\frac{k+2}{2\,k!} = 3/2 = a_2,
\]
so the result follows.
\end{proof}

\begin{lemma}				\label{lem:bnU}
If $n\in\Z$, $n \ge 1$, and $b_n$ is as above, then
\begin{equation}			\label{eq:bnU}
b_n = -\Gamma(n)\, U(n,0,1).
\end{equation}
\end{lemma}

\begin{proof}
We start with~\cite[(6.7.1)]{DLMF}:
\[
I(a,b):=
 \int_0^\infty \frac{ e^{-at}}{t+b}{\dup}t = e^{ab}E_1(ab),\,\,\,a,b > 0.\]
Note that, by definition, 
$b_n = [z^n]I(1,1/(1-z))$.
Setting $a=1$, $b=1/(1-z),$ the term $1/(t+b)$ inside the integral
can be rearranged 
as follows: 
\begin{equation*}
\left ( t+\frac{1}{1-z} \right)^{-1} = \frac{1-z}{1+t-tz}
=\frac{1}{1+t}-\frac{1}{t(1+t)}\left(\frac{1}{1-{zt}/(1+t)} -1 \right),
\end{equation*}
and making the substitution $s=t/(1+t)$ gives  
\[I(1,1/(1-z))= \int_0^\infty
\frac{e^{-t}}{1+t}{\dup}t - 
 \int_0^1 e^{-s/(1-s)}\, \left(\frac{z}{1-zs}\right){\dup}s
 = \sum_{n \ge 0} b_n \, z^n.\] 
Thus, $b_0=eE_1(1)$ and, for $n > 0$,
\begin{equation}			\label{eq:Larrybn1}
b_n=-\int_0^1 e^{-s/(1-s)}\, s^{n-1}{\dup}s.
\end{equation}
Writing $e^{-s/(1-s)} = e^{1-1/(1-s)}$ gives, for $n > 0$,
\begin{equation}			\label{eq:Larrybn2}
b_n=-e \int_0^1 e^{-1/(1-s)}\, {s^{n-1}}{\dup}s.
\end{equation}
Substitute $t=s/(1-s)$ in~\eqref{eq:NIST13.4.4}, giving
\begin{equation}				\label{eq:GUabz}
\Gamma(a)U(a,b,z)= e^z\int_0^1 e^{-z/(1-s)} \, s^{a-1} \,
(1-s)^{-b}{\dup}s.
\end{equation}
Comparison of~\eqref{eq:Larrybn2} and~\eqref{eq:GUabz}
now gives $b_n = -\Gamma(n)\, U(n,0,1)$.
\end{proof}
\begin{remark}
We could prove Lemma~\ref{lem:bnU} in the same manner as
Lemma~\ref{lem:anM},
using the connection formula~\eqref{eq:Kummer13.3.7}
instead of~\eqref{eq:Kummer13.3.1},
and the recurrence~\eqref{eq:brec1} instead of~\eqref{eq:arec1},
but in order to verify the initial conditions we would have to resort to
some explicit representation for $U$, such as the integral
representation~\eqref{eq:NIST13.4.4}, so the proof would be no simpler.
\end{remark}

\begin{remark}					\label{rem:abgen}
We can general{\ise} our definitions of $a_n$ and $b_n$ to permit $n\in\C$,
using Lemmas~\ref{lem:anM}--\ref{lem:bnU}.
Such generalizations do not seem particularly useful,
so in what follows we continue to assume that $n\in\Z$.
\end{remark}

\section[Asymptotic expansions of a(n) and b(n)]
{Asymptotic expansions of $a_n$ and $b_n$}
\label{sec:asymp_a_b}

Theorem~\ref{thm:ckdirect} gives the complete asymptotic expansions of $a_n$
and $b_n$ in ascending powers of $n^{-1/2}$. Wright~\cite{Wright0} proved
the existence of an asymptotic expansion of the form~\eqref{eq:aseries2} for
$a_n$, but did not state an explicit formula or algorithm for computing the
constants $c_m$ occurring in the expansion.
For a more ``algorithmic'' approach, see Wyman~\cite{Wyman59}.

\begin{theorem}				\label{thm:ckdirect}
For positive integer $n$, if $a_n$ and $b_n$ are as above, then
\begin{equation}			\label{eq:aseries2}
a_n \sim \frac{e^{2\sqrt{n}}}{2n^{3/4}\sqrt{\pi e}}
	\sum_{m\ge 0} c_m n^{-m/2}
\end{equation}
and
\begin{equation}			\label{eq:bseries2}
b_n \sim -\,\frac{\sqrt{\pi e}}{n^{3/4}e^{2\sqrt{n}}}
	\sum_{m\ge 0} (-1)^m c_m n^{-m/2},
\end{equation}
where
\begin{equation}			\label{eq:ckdirect}
c_m = (-1)^m\sum_{j=0}^m \, [h^{m-j}]\exp(\mu(h))\;
	\frac{(m-2j+3/2)_{2j}}{4^j j!}
\end{equation}
and
\begin{equation}			\label{eq:mu2}
\mu(h) = h^{-1} - (e^h-1)^{-1} - {\textstyle\half}\,\raisedot
\end{equation}
\end{theorem}
\begin{remark}
The function $\mu(h)$ defined by~\eqref{eq:mu2} could also be
defined using Bernoulli numbers, since
\begin{equation}			\label{eq:mu-Bernoulli}
\mu(h) = -\sum_{k=1}^\infty \frac{B_{2k}}{(2k)!} h^{2k-1}
 = -\frac{h}{12} + \frac{h^3}{720} - O(h^5).
\end{equation}
The function $\exp(\mu(h))$ occurring in~\eqref{eq:ckdirect} 
has the Maclaurin expansion
\begin{equation}			\label{eq:mu}
\exp(\mu(h)) = 1 - \frac{h}{12} + \frac{h^2}{288} 
		+ \frac{67h^3}{51840} + O(h^4).
\end{equation}
The numerators and denominators of the coefficients $[h^n]\exp(\mu(h))$  
have been added to the OEIS as \seqnum{A321937} and \seqnum{A321938},
respectively.
\end{remark}
\begin{proof}[Proof of Thm.~$\ref{thm:ckdirect}$]
We first prove~\eqref{eq:bseries2}. {From} Lemma~\ref{lem:bnU},
$b_n = -\Gamma(n)\, U(n,0,1)$.
Temme~\cite[Sec.~3]{Temme13} gives a general 
asymptotic result for $U(a,b,z^2)$ as $a \to \infty$. 
We state Temme's result for the case
$(a,b,z) = (n,0,1)$, 
which is what we need.
Let $c_k' := [h^k]\exp(\mu(h))$. (Temme uses $c_k$, but this conflicts with
our notation.)
{From} Temme~\cite[(3.8)--(3.10)]{Temme13}, we have
\begin{equation}				\label{eq:T3.8}
U(n,0,1) \sim \frac{\sqrt{e}}{\Gamma(n)}\sum_{k\ge 0} c_k'\Phi_k(n),
\end{equation}
where
\[\Phi_k(n) = 2n^{-(k+1)/2}K_{k+1}(2n^{1/2}),\]
and $K_\nu$ denotes the usual modified Bessel function.

{From}~\cite[(10.40.2)]{DLMF},
$K_\nu(z)$ has an asymptotic expansion
\begin{equation}			\label{eq:Kasymp}
K_\nu(z) \sim e^{-z}\sqrt{\frac{\pi}{2z}}
		\sum_{j \ge 0} \frac{(\nu-j+1/2)_{2j}}{j!\,(2z)^j}
		\,\raisedot
\end{equation}
Setting $\nu = k$ and $z=2n^{1/2}$ in~\eqref{eq:Kasymp}, we obtain
\[
\Phi_{k-1}(n) = 2n^{-k/2}K_k(2n^{1/2})
	\sim \frac{\sqrt{\pi}e^{-2\sqrt{n}}}{n^{1/4}}
	  \sum_{j\ge 0}\frac{(k-j+1/2)_{2j}}{j!\,4^j\,n^{(j+k)/2}}
	\,\raisedot
\]
Substituting this expression into~\eqref{eq:T3.8},
and grouping like powers of $n$, we obtain
\[
b_n 
 = -\Gamma(n)\,U(n,0,1)
 \sim -\frac{\sqrt{\pi e}}{n^{3/4} e^{2\sqrt{n}}}
	\sum_{m \ge 0} \sum_{j=0}^m
	  \frac{c_{m-j}'\, (m-2j+3/2)_{2j}}{j!\,4^j\,n^{m/2}}
	\,\raisedot  
\]
Now, comparison with~\eqref{eq:bseries2} shows that
\[
(-1)^m c_m = \sum_{j=0}^m \frac{c_{m-j}'\, (m-2j+3/2)_{2j}}
	{j!\,4^j}\,\raisecomma
\]
which completes the proof of~\eqref{eq:bseries2}.

The proof of~\eqref{eq:aseries2} is similar. We use Lemma~\ref{lem:anM}
instead of Lemma~\ref{lem:bnU}, and Temme's asymptotic 
result~\cite[(3.29)]{Temme13} for
$M(a,b,z^2)$ as $a \to \infty$
instead of~\eqref{eq:T3.8}; the modified Bessel function
$I_\nu$ replaces $K_\nu$.
{From}~\cite[(10.40.1)]{DLMF},
$I_\nu(z)$ has an asymptotic expansion
\begin{equation}                        \label{eq:Iasymp}
I_\nu(z) \sim \frac{e^{z}}{\sqrt{2\pi z}}
                \sum_{j \ge 0} (-1)^j\, \frac{(\nu-j+1/2)_{2j}}{j!\,(2z)^j}
                \,\raisecomma
\end{equation}
which replaces~\eqref{eq:Kasymp}.
\end{proof}

Theorem~\ref{thm:ckdirect} gives an expression for $c_m$ which (indirectly)
involves Bernoulli numbers, in view of~\eqref{eq:mu-Bernoulli}. 
Lemma~\ref{lemma:ck} gives a different expression
for $c_m$ that is recursive, as the expression for $c_m$ depends on the
values of $c_j$ for $j < m$, but has the advantage of avoiding reference to 
Bernoulli numbers. The idea of the proof is similar to that used
in the ``method of Frobenius''~\cite{Frobenius}.
We omit the details, which may be found
in~\cite[pp.~10--11]{BGG-arXiv}.

\begin{lemma}					\label{lemma:ck}
We have $c_0 = 1$ and, for all $m \ge 1$,
\begin{equation}		\label{eq:ckrecursion}
m c_m = [h^{m+3}]\,
	\sum_{j=0}^{m-1} c_j h^j\!
        \sum_{s \in \{\pm 1\}}
	(1+sh^2)^{\frac{1-2j}{4}}
	\exp\left(\frac{2}{h}\left((1+sh^2)^{\half}-1\right)\right).
\end{equation}
\end{lemma}

\begin{remark}
Computation using~\eqref{eq:ckdirect} and, as a
check, \eqref{eq:ckrecursion}, gives
\[ 
(c_k)_{k\ge 0} = \left(1, -\frac{5}{48}, 
	-\frac{479}{4608}, -\frac{15313}{3317760},
	\frac{710401}{127401984},
	-\frac{3532731539}{214035333120},
	 \ldots\right).\]
The numerators and denominators have been added to the OEIS
as \seqnum{A321939} and \seqnum{A321940}, respectively.
With the exception of $c_0$ and $c_4$, the $c_k$ all appear to
be negative.  This has been verified numerically for $k \le 1000$.
\end{remark}

\section[Maclaurin coefficients a(n) and b(n)]
{The Maclaurin coefficients $a_n$ and $b_n$}
\label{sec:Maclaurin}

The function $f_0(z)$ is the exponential generating function counting several
combinatorial objects, such as the number of ``sets of lists'',
i.e., the number of partitions of $\{1,2,\ldots,n\}$ into ordered subsets,
see Wallner~\cite[\S5.3]{Wallner}.

Observe that $f_0(z)$ satisfies the differential equation
\begin{equation}
(1-z)^2 f_0'(z) - f_0(z) = 0,		\label{eq:f0de}
\end{equation}
and from this it is easy to see that the $a_n$ satisfy a three-term recurrence
\begin{equation}			\label{eq:arec1}
na_n - (2n-1)a_{n-1} + (n-2)a_{n-2} = 0 \;\text{ for }\; n \ge 2.
\end{equation}
The initial conditions are $a_0 = a_1 = 1$.
Thus
\[(a_n)_{n\ge 0} = (1, 1, 3/2, 13/6, 73/24, 167/40, \ldots).\]
The recurrence~\eqref{eq:arec1} holds for $n \ge 0$ provided that we
define $a_n = 0$ for $n < 0$.
A closed-form expression, valid for $n \ge 1$ (but not for $n=0$), is
\[
a_n = \sum_{k=1}^n \frac{1}{k!}\,\binom{n-1}{k-1}.
\]

The constants $a_n$ may be expressed in terms of the general{\ise}d Laguerre
polynomials $L_n^{(\alpha)}(x)$
which, from~\cite[(18.12.13)]{DLMF}, have a generating function
\[\sum_{n\ge 0} z^n L_n^{(\alpha)}(x) = (1-z)^{-(\alpha+1)}e^{-xz/(1-z)}.\]
With $\alpha = x = -1$ we obtain
$\sum_{n\ge 0} z^n L_n^{(-1)}(-1) = e^{z/(1-z)}$, so
$a_n = L_n^{(-1)}(-1)$.

Using the chain rule and the definition of $f_1(z)$ in \S\ref{sec:intro},
we see that $f_1(z)$
satisfies the differential equation
\begin{equation}				\label{eq:f1de}
(1-z)^2 f_1'(z) - f_1(z) = z-1,
\end{equation}
which differs from~\eqref{eq:f0de} only in the right-hand side $z-1$.
Differentiating twice more with respect to $z$, 
we see that $f_0(z)$ and $f_1(z)$
both satisfy the same third-order differential equation
\[(1-z)^2f''' + (4z-5)f'' + 2f' = 0.\]

{From}~\eqref{eq:f1de}, the $b_n$ satisfy a recurrence
\begin{equation}			\label{eq:brec1}
nb_n - (2n-1)b_{n-1} + (n-2)b_{n-2} =
 \begin{cases}
 1, & \text{if $n=2$;}\\
 0, & \text{if $n \ge 3$.}
 \end{cases} 
\end{equation}
This is essentially
(i.e., for $n \ge 3$)
the same recurrence as~\eqref{eq:arec1}, but the initial conditions
$b_0 = G$, $b_1 = G-1$
are different. 
Here $G := e\Eone(1) \approx 0.596$ is the Euler-Gompertz
constant~\cite[\S2.5]{Lagarias}.

We remark that computation of the $b_n$ using the
recurrence~\eqref{eq:brec1} in the forward direction is
numerically unstable. A stable method of computation
is to use an adaptation of Miller's algorithm,
originally used to compute Bessel functions.
See Gautschi~\cite[\S3]{Gautschi} and Temme~\cite[\S4]{Temme75}.

As noted in \S\ref{sec:intro}, the $b_n$ may be expressed as
$a_nG-a_n'$, where $a_n$ is as above, and $a_n'$ satisfies essentially
the same recurrence with different initial conditions. In fact,
\[
na_n' - (2n-1)a_{n-1}' + (n-2)a_{n-2}' =
 \begin{cases}
 -1, & \text{if $n = 2$;}\\
  \phantom{-}0, &\text{if $n \ge 3$.}
 \end{cases}
\]
The initial conditions are $a_0' = 0$, $a_1' = 1$.
Thus
\[(a_n')_{n\ge 0} = (0, 1, 1, 4/3, 11/6, 5/2, 121/36, \ldots).\]

{From} \eqref{eq:bseries2}, $b_n \to 0$ as $n \to \infty$,
so the sequence $(a_n'/a_n)_{n\ge 1}$ is a convergent sequence of rational
approximations to~$G$. The sequence of approximants is
$(1, 2/3, 8/13, 44/73, 100/167, \ldots)$.

Bala~\cite{Bala} gives the continued fraction
\[1-G = 1/(3 - 2/(5 - 6/(7 - \cdots -n(n+1)/(2n+3) - \cdots))),\]
with convergents $1/3, 5/13, 28/73, 201/501$, etc.
The corresponding convergents to $G$ are
$2/3, 8/13, 45/73, 100/167$, etc.  We see that the $n$-th convergent
is just $a_{n+1}'/a_{n+1}$.
Theorem~\ref{thm:ckdirect} implies that
\[
G - a_n'/a_n = b_n/a_n \sim -2\pi e^{1-4\sqrt{n}} \text{ as } n \to \infty.
\]
We have contributed the sequence $(n!a_n')_{n\ge 1}$ to the OEIS as
\seqnum{A321942}.

\section[The Hadamard product of f0 and f1]
	{The Hadamard product of $f_0$ and $f_1$}	\label{sec:Hadamard}
	
Define $\rho_n := a_n b_n$. Thus $\sum_{n=0}^\infty \rho_n z^n$
is the Hadamard product $(f_0\Had f_1)(z)$.
{From} Lemmas~\ref{lem:anM}--\ref{lem:bnU}, we have
\[
\rho_n = -e^{-1}\Gamma(n)M(n+1,2,1)U(n,0,1).
\] 
Using Theorem~\ref{thm:ckdirect}, we can obtain a complete asymptotic expansion
for $\rho_n$ in decreasing powers of~$n$. This is given in
Corollary~\ref{cor:product}.

\begin{corollary}				\label{cor:product}
We have
\[
\rho_n \sim -\,\frac{1}{2n^{3/2}}\sum_{k\ge 0} d_k n^{-k},
\]
where
\[d_k = \sum_{j=0}^{2k}(-1)^j c_j c_{2k-j},\]
and $c_0,\ldots,c_{2k}$ are as in Theorem~$\ref{thm:ckdirect}$.
\end{corollary}

A computation shows that
\[(d_k)_{k\ge 0} = (1, -7/32, 43/2048, -915/65536, \ldots).\]
We observe that the $d_k$ appear to be dyadic rationals
More precisely, it appears that $2^{6k}d_k \in \Z$.
Define a scaled sequence $(r_k)_{k\ge 0}$ by
$r_k := 2^{6k}d_k$. Computation gives
\[
(r_k)_{k\ge 0} = (1, -14, 86, -3660 , -1042202, -247948260,
		-108448540420, \ldots).
\]
This leads naturally to the following conjecture.

\pagebreak[3]

\begin{conjecture}			\label{conj:rkinteger}
For all $k \ge 0$, $r_k\in\Z$. 
\end{conjecture}
The sequence of numerators of $r_k$ 
has been added to the OEIS as \seqnum{A321941}.
If Conjecture \ref{conj:rkinteger} holds, then the denominators
are all~$1$, i.e., the denominators are given by \seqnum{A000012}.

\begin{remark}				\label{remark:rkinteger}
Conjecture~\ref{conj:rkinteger} has been verified for all $k \le 1000$.
We also showed numerically, for $3 \le k \le 1000$,
that $r_k < 0$ and
$r_k \equiv \binom{2k}{k}$ 
(mod $32$).	
\end{remark}
\begin{remark}
A problem that is superficially
similar to our conjecture was solved by Tulyakov~\cite{Tulyakov}.
However, we do not see how to adapt his method
to prove our conjecture.
\end{remark}
\begin{remark}				\label{remark:Bessel}
Corollary~\ref{cor:product} is reminiscent of the result
\[
I_0(x) K_0(x) \sim \frac{1}{2x}\sum_{k\ge 0} e_{k,0}\, x^{-2k}
\]
in the theory of Bessel functions~\cite[(1.2)]{rpb256}.
The coefficients $e_{k,0}$ are given by
\[
e_{k,0} = \frac{(2k)!^3}{2^{6k}k!^4}\,\raisecomma
\]
so $2^{4k}e_{k,0}\in\Z$.
The modified Bessel functions $I_0(x)$ and
$K_0(x)$ are solutions of the same ordinary differential
equation
$xy'' + y' -xy = 0$,
but $I_0(x)$ increases with $x$ 
while $K_0(x)$ decreases.
This is analogous to the behaviour of $a_n$, which increases
as $n \to \infty$, and $|b_n|$, which decreases as $n \to \infty$.

More generally, from \cite[(10.40.6)]{DLMF},
we have
\[I_\nu(x)K_\nu(x) \sim \frac{1}{2x}\sum_{k\ge 0} e_{k,\nu}x^{-2k},\]
where
\[
e_{k,\nu} = (-1)^k 2^{-2k}(\nu-k+1/2)_{2k}\binom{2k}{k},
\]
and $2^{4k}e_{k,\nu}\in\Z$ for $\nu\in\Z$.
\end{remark}

\section[Other expressions for d(n)]
	{Other expressions for $d_n$}		\label{sec:dn_rec}

Since $(a_n)$
and $(b_n)$ are D-finite, it follows that $(\rho_n)$ is D-finite.%
\footnote{
See Flajolet and Sedgewick~\cite[Appendix B.4]{FS}, and
Stanley~\cite[Theorem~2.10]{Stanley},
for relevant background on D-finite sequences.}
In fact, $\rho_n$
satisfies the $4$-term recurrence
\begin{align}
\nonumber
n^2(n-1)(2n-3)\rho_n =& \;\;(n-1)(2n-1)(3n^2-5n+1)\rho_{n-1}\\
\nonumber
	&-(n-2)(2n-3)(3n^2-5n+1)\rho_{n-2}\\
	&+ (n-2)(n-3)^2(2n-1)\rho_{n-3}		\label{eq:rho-rec1}
\end{align}
for $n\ge 3$, with initial conditions 
$\rho_0 = G$, $\rho_1 = G-1$, $\rho_2 = (9G-6)/4$.

The recurrence~\eqref{eq:rho-rec1}
can be simplified by defining $\sigma_n := n\rho_n$. Then $\sigma_n$
satisfies the slightly simpler recurrence
\begin{align}
\nonumber
n(n-1)&(2n-3)\sigma_n = (2n-1)(3n^2-5n+1)\sigma_{n-1}\\
\label{eq:sigma-rec1}
	&-(2n-3)(3n^2-5n+1)\sigma_{n-2} + (n-2)(n-3)(2n-1)\sigma_{n-3}		
\end{align}
for $n\ge 3$, with initial conditions 
$\sigma_0 = 0$, $\sigma_1 = G-1$, $\sigma_2 = 9G/2-3$.
Also, Corollary~\ref{cor:product} gives an asymptotic series for $\sigma_n$:
\begin{equation}				\label{eq:sigma-asymp}
\sigma_n \sim -\,\frac{1}{2n^{1/2}}\sum_{k\ge 0} d_k n^{-k}.
\end{equation}

Using~\eqref{eq:sigma-rec1}, we can give a recursive algorithm for
computing the sequence $(d_n)$ (and hence $(r_n)$) directly, without
computing the sequence $(c_n)$.

\begin{lemma}				\label{lemma:dn_direct2}
We have $d_0 = 1$ and, for all $k \ge 1$,
\begin{align}
\nonumber
8kd_k = -\,[h^{k+2}]\, &\Bigg(\sum_{j=0}^{k-1}
	  d_j h^j \bigg(
	  B(h)(1-h)^{-(j+1/2)} \\
\label{eq:dn_direct1a}
	 &+ C(h)(1-2h)^{-(j+1/2)} + D(h)(1-3h)^{-(j+1/2)}\bigg)\Bigg),
\end{align}
where
\begin{align*}
B(h) &= -6 + 13h - 7h^2 + h^3 \;\;\;\, = -(2-h)(3-5h+h^2),\\
C(h) &= +6 - 19h + 17h^2 - 3h^3 = (2-3h)(3-5h+h^2), \;\text{ and }\\
D(h) &= -2 + 11h - 17h^2 + 6h^3 = -(1-2h)(1-3h)(2-h).
\end{align*}
\end{lemma}
\begin{proof}
Define $h := n^{-1}$, so $h \to 0$ as $n \to \infty$.
{From} Corollary~\ref{cor:product}, there exists an asymptotic series
of the form
\[-2\sigma_n \sim \sum_{j \ge 0} d_j n^{-j-1/2}\]
as $n \to \infty$.  Moreover, $d_0 = 1$.
Define $A(h) := (1-h)(2-3h)$ in addition to $B(h)$, $C(h)$ and $D(h)$.
Using the recurrence~\eqref{eq:sigma-rec1}
and the elementary identity 
$1/(n-m) = h/(1-mh)$ for $m \in \{0,1,2,3\}$, we have
\begin{align*}
\sum_{j\ge 0} d_j \Bigg(&A(h)h^{j+1/2}
	+ B(h)\left(\frac{h}{1-h}\right)^{j+1/2}\\
	+ &\;C(h)\left(\frac{h}{1-2h}\right)^{j+1/2}
	+ \;D(h)\left(\frac{h}{1-3h}\right)^{j+1/2}\Bigg) \sim 0.
\end{align*}
Now, dividing both sides by $h^{1/2}$, we obtain
\begin{align}
\nonumber
\sum_{j\ge 0}
	  d_j h^j \bigg(&A(h) + B(h)(1-h)^{-(j+1/2)} \\
\label{eq:dn_direct2}
	 &+ C(h)(1-2h)^{-(j+1/2)} + D(h)(1-3h)^{-(j+1/2)}\bigg) \sim 0.
\end{align}
An easy computation shows that
\begin{align*}
A(h) + B(h) + C(h) + D(h) &= -4h^2 + O(h^3),\\
       B(h) + 2C(h) + 3D(h) &= 8h + O(h^2),\;\text{ and}\\
       B(h) + 2^2C(h) + 3^2D(h) &= O(h).
\end{align*}
Thus, for all $j \ge 1$, the terms involving $d_j$ in~\eqref{eq:dn_direct2}
are $8jh^{j+2} + O(h^{j+3})$. 
(The ``$8j$'' arises from $-4 + 8(j+1/2) = 8j$.)
This shows that the choice of $d_k$
in~\eqref{eq:dn_direct1a} is necessary and sufficient to give an asymptotic
series of the required form.
Finally, we note that
$[h^{k+2-j}]A(h) = 0$, since $j \le k-1$ and $\deg(A(h)) = 2$.
Thus, a term involving $A(h)$ has been omitted from~\eqref{eq:dn_direct1a}.
\end{proof}

Using Lemma~\ref{lemma:dn_direct2}, we computed
the sequences $(d_n)$ and $(r_n)$ for $n \le 1000$, and verified
the values previously computed (more slowly) via Corollary~\ref{cor:product}.

Since the power series occurring in~\eqref{eq:dn_direct1a} have a simple form,
we can extract the coefficients of the required powers of $h$ 
to obtain a recurrence for the $d_k$, as in Corollary~\ref{cor:explicit_dn2}.
This gives a third way to compute the sequence $(d_n)$.

\begin{corollary}			\label{cor:explicit_dn2}
We have $d_0 = 1$ and, for all $k \ge 1$,
\[
8k\, d_k = \sum_{j=0}^{k-1} \alpha_{j,k}\,d_j.
\]
Here
\vspace*{-10pt}
\begin{align}
\nonumber
\alpha_{j,k} =&
 \;\;(-1 + 3\cdot 2^{m-1} - 2\cdot 3^m) (\tau)_{m-1}/(m-1)!\\
\nonumber
 &+ (7 - 17\cdot 2^m + 17\cdot 3^m) (\tau)_{m}/m!\\
\nonumber
 &+ (-13 + 38\cdot 2^m - 33\cdot 3^m) (\tau)_{m+1}/(m+1)!\\
\label{eq:alphajk2}
 &+ 6(1 - 4\cdot 2^m + 3\cdot 3^m) (\tau)_{m+2}/(m+2)!,
\end{align}
where $m:= k-j$ and $\tau := j+1/2$.
\end{corollary}
\begin{proof}[Proof (sketch)]
To prove Corollary~\ref{cor:explicit_dn2}, we apply the binomial
theorem to the power series in~\eqref{eq:dn_direct1a}, multiply by the
polynomials $B(h)$, $C(h)$, and $D(h)$, and  extract
the coefficient of $h^{k+2-j}$.
\end{proof}

The following corollary is an easy deduction from
Corollary~\ref{cor:explicit_dn2}, and gives an explicit recurrence for
$r_k = 2^{6k}d_k$.

\begin{corollary}			\label{cor:explicit_rk}
We have $r_0 = 1$ and, for all $k \ge 1$,
\[k\,r_k = \sum_{j=0}^{k-1} \beta_{j,k}\,r_j,\;
\text{ where }\; \beta_{j,k} = 8^{2k-2j-1}\,\alpha_{j,k}\,.\]
\end{corollary}

Although we have not proved Conjecture~\ref{conj:rkinteger},
the following result goes part of the way.

\begin{theorem}				\label{thm:factorialRkk}
For all $k \ge 0$, we have $k!\,r_k \in \Z$. 
\end{theorem}
\begin{proof}
Let $R_k := k!r_k$.  We show that $R_k\in\Z$.
{From} Corollary~\ref{cor:explicit_rk},
$R_0 = 1$ and, for $k \ge 1$, $R_k$ satisfies the recurrence
\begin{equation}			\label{eq:Rkrec}
R_k = \sum_{j=0}^{k-1} \beta_{j,k}\,R_j\,\frac{(k-1)!}{j!}\,\raisedot
\end{equation}
The ratio of factorials in~\eqref{eq:Rkrec} is an integer, since 
$j \le k-1$. Thus, in order to prove the result by induction  on $k$,
it is sufficient to show that $\beta_{j,k} \in\Z$.
Now, elementary number theory shows that 
$4^\ell(j+1/2)_\ell/\ell! \in \Z$
for all $j,\ell \ge 0$.
Thus, the expressions of the form $(\tau)_{m+\delta}/(m+\delta)!$
in~\eqref{eq:alphajk2} are in $\Z$ provided that $m+\delta \ge 0$.
This is true as $m \ge k-j \ge 1$ and $\delta \ge -1$.
To show
that $\beta_{j,k}\in\Z$, it is sufficient to have $8^{2m-1} \ge 4^{m+2}$,
which holds for all $m \ge 2$. In the case $m=1$, it is easy to see
that all the terms in~\eqref{eq:alphajk2} are in~$\Z/4$,
so $\beta_{m-1,k} = 8\alpha_{m-1,k} \in\Z$.
Thus, $\beta_{j,k}\in\Z$ for $0 \le j < k$, 
and the result follows by induction on~$k$.
\end{proof}
\begin{remark}
The proof actually shows that $\beta_{j,k}\in 2\Z$, which implies
that $R_k \in 2\Z$ for all $k > 0$. 
\end{remark}

\section{Acknowledgments}

We thank Bruno Salvy for communicating his conjecture to one of us.
An anonymous referee made helpful suggestions regarding the exposition.
RPB was supported in part by ARC grant DP140101417.
AJG wishes to acknowledge support of the ARC Centre of Excellence for
Mathematical and Statistical Frontiers (ACEMS).

\pagebreak[3]

\pagebreak[3]

\bigskip
\hrule
\bigskip

\noindent 2010 {\it Mathematics Subject Classification}:
Primary 
34E05;			
Secondary 
11Y55,			
33C10,			
33C15, 			
33F99.			

\noindent \emph{Keywords: }
asymptotics,
confluent hypergeometric function,
D-finite,
Euler-Gompertz constant,
exponential integral,
Hadamard product,
holonomic,
Kummer function,
modified Bessel function,
stretched exponential.

\bigskip
\hrule
\bigskip

\noindent (Concerned with sequences
\seqnum{A000012},	
\seqnum{A000262},	
\seqnum{A067653},	
\seqnum{A067764},	
\seqnum{A073003},	
\seqnum{A201203},	
\seqnum{A293125},	
\seqnum{A321937},	
\seqnum{A321938},	
\seqnum{A321939},	
\seqnum{A321940},	
\seqnum{A321941},	
and			
\seqnum{A321942}.)	

\bigskip
\hrule
\bigskip

\vspace*{+.1in}
\noindent
\noindent

\end{document}